\documentclass[12pt,reqno,oneside]{amsart}
\usepackage{amsmath,amsthm}
\usepackage[papersize={210mm,297mm},margin=2.5cm,top=3cm,bottom=3.5cm,footskip=1cm,nohead]{geometry}\def\DOIsuffix#1{}\def\Volume#1{}\def\Month#1{}\def\Year#1{}
\def\pagespan#1#2{}\def\Receiveddate#1{}
\def\Reviseddate#1{}\def\Accepteddate#1{}\def\Dateposted#1{}
\let\tmpsubjclass\subjclass\def\subjclass[#1]#2{\tmpsubjclass[2000]{#2}}
\newtheorem{theorem}{Theorem}[section]        

\newtheorem{lem}[theorem]{Lemma}
\theoremstyle{definition}\newtheorem{definition}[theorem]{Definition}
\newtheorem{example}[theorem]{Example}
\theoremstyle{remark}\newtheorem{remark}[theorem]{Remark}
\newtheorem*{acknowledgement}{Acknowledgment}

\usepackage[psamsfonts]{amssymb}
\makeatletter
\let\tmpmaketitle\maketitle
\def\maketitle{\tmpmaketitle\thispagestyle{plain}}
\makeatother
\numberwithin{equation}{section}
\makeatletter
\newcommand{\rom}[1]{\mbox{\leavevmode\skip@\lastskip\unskip\/%
           \ifdim\skip@=\z@\else\hskip\skip@\fi{\rm{#1}}}}
\makeatother
\newcommand{\Thm}[1]{Theorem~\ref{th:#1}}

\newcommand{\Lem}[1]{Lemma~\ref{lem:#1}}

\newcommand{\Eq}[1]{\rom{\eqref{eq:#1}}}
\renewcommand{\a}{\alpha}
\newcommand{\dl}{\delta}

\newcommand{\kp}{\kappa}
\newcommand{\sg}{\sigma}
\newcommand{\ph}{\varphi}
\newcommand{\la}{\langle}\newcommand{\ra}{\rangle}
\def\cB{\mathcal{B}}
\def\cE{\mathcal{E}}
\def\cF{\mathcal{F}}
\newcommand{\cFloc}{\dot{\mathcal{F}}_{\mathrm{loc}}}
\def\cG{\mathcal{G}}
\def\cM{\mathcal{M}}
\newcommand{\bfa}{{\boldsymbol a}}
\newcommand{\bfb}{{\boldsymbol b}}

\newcommand{\bfg}{{\boldsymbol g}}
\newcommand{\bfh}{{\boldsymbol h}}

\newcommand{\bfu}{{\boldsymbol u}}
\newcommand{\bfv}{{\boldsymbol v}}
\newcommand{\bfx}{{\boldsymbol x}}
\newcommand{\bfz}{{\boldsymbol 0}}
\newcommand{\bfone}{{\boldsymbol 1}}
\newcommand{\bfM}{{\boldsymbol M}}

\newcommand{\N}{\mathbb{N}}
\newcommand{\R}{\mathbb{R}}
\newcommand{\Z}{\mathbb{Z}}
\newcommand{\maruM}{{\stackrel{\circ}{\smash{\mathcal{M}}\rule{0pt}{1.3ex}}}}
\newcommand{\rank}{\mathop{\mathrm{rank}}\nolimits}
\newcommand{\supp}{\mathop{\mathrm{supp}}\nolimits}
\newcommand{\nuesssup}{\mathop{\nu\mathrm{\mbox{-}\,esssup}}}
\def\tr#1{\mathord{\mathopen{{\vphantom{#1}}^t}#1}}

\allowdisplaybreaks[3]
\begin{document}
\DOIsuffix{mana.DOIsuffix}
\Volume{}
\Month{}
\Year{}
\pagespan{1}{}
\Receiveddate{}
\Reviseddate{}
\Accepteddate{}
\Dateposted{}


\title[Measurable Riemannian structures]{Measurable Riemannian structures associated with strong local Dirichlet forms}

\author[M. Hino]{Masanori Hino}%
\address{Graduate School of Informatics, Kyoto University, Kyoto 606-8501, Japan}
\email{hino@i.kyoto-u.ac.jp}
\thanks{Research partially supported by KAKENHI (21740094, 24540170).}
\begin{abstract}
We introduce Riemannian-like structures associated with strong local Dirichlet forms on general state spaces. Such structures justify the principle that the pointwise index of the Dirichlet form represents the effective dimension of the virtual tangent space at each point.
The concept of differentiations of functions is studied, and an application to stochastic analysis is presented.
\end{abstract}
\maketitle                   




\section{Introduction}
Thus far, numerous studies have been conducted to investigate local structures that are derived from diffusion processes and to obtain analytic information that reflects the local behavior of the processes. Although these studies are fundamental, general theories have been often proposed in rather limited frameworks.
In a typical situation where the state space has a Riemannian  structure and the diffusion process is given by a suitable stochastic differential equation, its generator, which is described as a second-order differential operator, is a solution to the first step of the problems.
However, there are many examples without such simple structures, e.g., Brownian motions on fractals and various kinds of singular diffusions on Euclidean spaces.
On the other hand, from different viewpoints, it has been observed that the investigation of the filtration associated with the diffusion process, in particular, a class of martingales with respect to the filtration, was useful for understanding the local structures. 
Some of the pioneering works include \cite{MW64,Sk66,KW67,IW73,DV74} and a series of papers published by Marc~Yor in the 1970s (also see the references therein).
Setting up the problem in this manner is valid in general situations;
for example, Kusuoka~\cite{Ku89} proved that the AF-martingale dimension, which represents a type of multiplicity of filtration, is always $1$ for  Brownian motion on an arbitrarily dimensional standard Sierpinski gasket.
The quantitative estimate concerning underlying spaces with anomalous structure of this type is a highly nontrivial problem.
Inspired by this work, the author introduced in \cite{Hi10} the analytic concept of the (pointwise) index of strong local regular Dirichlet forms $(\cE,\cF)$ on general state spaces, and proved that the index coincides with the AF-martingale dimension of the diffusion process associated with $(\cE,\cF)$.
This characterization was used in \cite{Hi11p} to deduce some estimates of AF-martingale dimensions for self-similar fractals, which generalized the results of \cite{Ku89,Hi08}.
Moreover, the concept of the derivative of functions in $\cF$ was studied in \cite{Hi10} for a class of self-similar fractals, which informally implied that the index represents the dimension of the proper ``tangent space'' of the underlying fractal set.

The main objective of this paper is to justify the above statement more explicitly in the general framework.
Given a strong local regular Dirichlet form $(\cE,\cF)$ on a general state space $X$ with finite index $p$, we prove that there exist $p$ functions $g_1,\dots,g_p$ in $\cF$ that play the role of a type of local coordinate system, and that every function $f$ in $\cF$ has a differentiation $\nabla\!_\bfg f$ with respect to $\bfg=\tr{(g_1,\dots,g_p)}$.
We also show that $\cE$ has an integral representation using $\nabla\!_\bfg$, analogous to the classical energy form.
From these results, we may say that a Riemannian-like structure associated with $(\cE,\cF)$ is equipped with $X$ and that the pointwise index is interpreted as the effective dimension of the ``tangent space'' at each point.
We denote such a structure by the {\em measurable Riemannian structure}, following the terminology in \cite{Ki08}.
As an application of these results to stochastic analysis, we improve upon the theorem in \cite{Hi10} on the stochastic integral representation of martingale additive functionals.

We should remark here that results similar to the claims stated above have been obtained in previous studies on the analysis on fractals: gradient operators on some fractals were introduced in \cite{Ku89,Ku93,Te00}, the concept of differentiation along a representative in $\cF$ was discussed in \cite{PT08,Hi10}, and the measurable Riemannian structure on the Sierpinski gasket was considered in \cite{Ki93,Ki08}.
Furthermore, in Chapter~3 of \cite{Eb}, a family of Hilbert spaces was introduced as the tangent bundle associated with a general strong local Dirichlet form.
The study in this paper differs from those stated above in that underlying spaces do not need particular structures, and the differentiation of functions is realized using the minimal number of functions; in other words, the effective ``Riemannian metric'' is nondegenerate almost everywhere. 
This refinement helps to clarify the intrinsic structure of the Dirichlet form.
We hope that such improvements will be useful for further investigation of the local structures of diffusion processes as well as the development of differential calculus on nonsmooth spaces, based on the theory of Dirichlet forms.

The remainder of this paper is organized as follows. In Section~2, we introduce some concepts of Dirichlet forms and provide a few examples. In Section~3, we prove two main theorems about (1)~the existence of a set of functions considered as a generalized local coordinate system and (2)~the differentiation formula on functions in the domain of the Dirichlet form. In Section~4, we discuss an application to stochastic analysis, using the results presented in the previous section.
\section{Index of strong local Dirichlet form}
First, we introduce some basic concepts of Dirichlet forms, following~\cite{FOT}.
Let $X$ be a locally compact, separable, and metrizable space. 
Let $m$ be a positive Radon measure on $X$ with full support.
For an $m$-measurable function $f$ on $X$, we denote the support of measure $|f|\cdot m$ by $\supp[f]$.
Let $(\cE,\cF)$ be a regular Dirichlet form on $L^2(X;m)$.
The set $\cF$ becomes a Hilbert space with the inner product $(f,g)_\cF:=\cE(f,g)+\int_X fg\,dm$ for $f,g\in\cF$.
We assume that $(\cE,\cF)$ is also strong local, that is, $\cE(f,g)=0$ for $f,g\in\cF$ if both $\supp[f]$ and $\supp[g]$ are compact and $g$ is constant on a neighborhood of $\supp[f]$.
We write $\cE(f)$ for $\cE(f,f)$.
Let $\cF_b$ denote the set of all bounded functions in $\cF$, and
$C_c(X)$, the space of all continuous functions on $X$ with compact support.
For each $f\in\cF$, we define the energy measure $\mu_{\la f\ra}$ on $X$ as follows~(\cite[Section~3.2]{FOT}).
If $f$ is bounded, $\mu_{\la f\ra}$ is determined by the identity
\[
  \int_X \ph\,d\mu_{\la f\ra}=2\cE(f\ph,f)-\cE(\ph,f^2)\quad
  \mbox{for all }\ph\in\cF\cap C_c(X).
\]
From the inequality
\begin{equation*}
  \left|\sqrt{\mu_{\la f\ra}(B)}-\sqrt{\mu_{\la g\ra}(B)}\right|^2\le \mu_{\la f-g\ra}(B)
  \le 2\cE(f-g)
\end{equation*}
for any Borel subset $B$ of $X$ and $f,g\in\cF_b$
(cf.~\cite[p.~123]{FOT}), we can define $\mu_{\la f\ra}$ for any $f\in\cF$ by the limiting procedure.
From the strong locality of $(\cE,\cF)$, the identity
\begin{equation}\label{eq:total}
\cE(f)=\frac12\mu_{\la f\ra}(X)
\end{equation}
holds for $f\in\cF$ (see \cite[Lemma~3.2.3]{FOT}).
For $f,g\in\cF$, the mutual energy measure $\mu_{\la f,g\ra}$, which is a signed Borel measure on $X$, is defined as
\begin{equation}\label{eq:mutual}
  \mu_{\la f,g\ra}=\frac12(\mu_{\la f+g\ra}-\mu_{\la f\ra}-\mu_{\la g\ra}).
\end{equation}
Then, $\mu_{\la f,f\ra}=\mu_{\la f\ra}$ and $\mu_{\la f,g\ra}$ is bilinear in $f$ and $g$.
Moreover, for $f,g\in\cF$ and any Borel subset $B$ of $X$,
\begin{equation}
\label{eq:schwarz}
\left|\mu_{\la f,g\ra}(B)\right|\le\sqrt{\mu_{\la f\ra}(B)}\sqrt{\mu_{\la g\ra}(B)}.
\end{equation}

Associated with $(\cE,\cF)$, there exists a diffusion process $\{X_t\}$ on the one-point compactification $X_\Delta$ of $X$ with a filtered probability space $(\Omega,\cF_\infty,P,\{P_x\}_{x\in X_\Delta},\{\cF_t\}_{t\in[0,\infty)})$.
An $m$-measurable function $f$ on $X$ is called locally in $\cF$ in the broad sense ($f\in \cFloc$ in notation) if there exist a sequence of nearly Borel finely open sets $\{G_n\}_{n=1}^\infty$ and a sequence $\{u_n\}_{n=1}^\infty$ in $\cF$ such that $G_n\subset G_{n+1}$ for every $n\in\N$, $\bigcup_{n=1}^\infty G_n=X$ q.e., and $f=u_n$ $m$-a.e.\ on $G_n$ for every $n\in\N$.
We can then define the energy measure $\mu_{\la f\ra}$ of $f\in\cFloc$ so that $\mu_{\la f\ra}|_{G_n}=\mu_{\la u_n\ra}|_{G_n}$ for every $n$.
For $f,g\in\cFloc$, a signed measure $\mu_{\la f,g\ra}$ on $X$ is defined as \Eq{mutual}, and inequality~\Eq{schwarz} holds as long as the total masses of $\mu_{\la f\ra}$ and $\mu_{\la g\ra}$ are both finite.
The constant function $\bfone$ on $X$ belongs to $\cFloc$ (cf.~\cite[Theorem~4.1]{Kuw98}), and its energy measure is a null measure.
In particular,
\begin{equation}\label{eq:constant}
\mu_{\la f+c\bfone\ra}=\mu_{\la f\ra}\quad
\text{for any $f\in\cF$ and $c\in\R$}.
\end{equation}

For two $\sg$-finite (or signed) Borel measures $\mu_1$ and $\mu_2$ on $X$, we write $\mu_1\ll\mu_2$ if $\mu_1$ is absolutely continuous with respect to $|\mu_2|$.
Following \cite{Hi10}\footnote{In \cite{Hi10}, the energy measure of $f$ is denoted by $\nu_f$. In this paper, we adopt the symbol $\mu_{\la f\ra}$, following \cite{FOT}.}, we introduce the concepts of {minimal energy-dominant measure} and {index} of $(\cE,\cF)$.
\begin{definition}\label{def:mdem}
A $\sg$-finite Borel measure $\nu$ on $X$ is called a {\em minimal energy-dominant measure} (m.\,e.\,d.\,m.) of $(\cE,\cF)$ if the following two conditions are satisfied.
\begin{enumerate}
\item[(a)] (Domination) For every $f\in \cF$, $\mu_{\la f\ra}\ll \nu$;
\item[(b)] (Minimality) If another $\sg$-finite Borel measure $\nu'$ on $X$ satisfies condition (a) with $\nu$ replaced by $\nu'$, then $\nu\ll\nu'$.
\end{enumerate}
\end{definition}
By definition, two m.\,e.\,d.\,m.'s are mutually absolutely continuous.
There always exists an m.\,e.\,d.\,m.\ (cf.\ \cite[Lemma~2.3]{Hi10}).
From \Eq{schwarz}, $\mu_{\la f,g\ra}\ll\nu$ for m.\,e.\,d.\,m.\ $\nu$ and $f,g\in\cF$.

Fix an m.\,e.\,d.\,m.\ $\nu$ of $(\cE,\cF)$.
Let $\Z_+$ denote the set of all nonnegative integers.
\begin{definition}\label{def:index}
The pointwise index and the index of $(\cE,\cF)$ are defined as follows.
\begin{enumerate}
\item The {\em pointwise index} $p(x)$ is a $\nu$-measurable function on $X$ taking values in $\Z_+\cup\{+\infty\}$ such that the following hold:
\begin{enumerate}
\item For any $N\in\N$ and any $f_1,\dots,f_N\in\cF$, 
\[
 \rank \left(\frac{d\mu_{\la f_i,f_j\ra}}{d\nu}(x)\right)_{i,j=1}^N\le p(x)\quad \mbox{for }\nu\mbox{-a.e.\,}x\in X;
\]
\item If another function $p'(x)$ satisfies (a) with  $p(x)$ replaced by $p'(x)$, then $p(x)\le p'(x)$ for $\nu$-a.e.\,$x\in X$.
\end{enumerate}
\item The {\em index} $p$ is defined as $p=\nuesssup_{x\in X}p(x)\in\Z_+\cup\{+\infty\}$.
In other words, $p$ is the smallest number satisfying the following:
for any $N\in\N$ and any $f_1,\dots,f_N\in\cF$, 
\[
 \rank \left(\frac{d\mu_{\la f_i,f_j\ra}}{d\nu}(x)\right)_{i,j=1}^N\le p\quad \mbox{for }\nu\mbox{-a.e.\,}x\in X.
\]
\end{enumerate}
\end{definition}
These definitions are independent of the choice of $\nu$.
The pointwise index $p(x)$ is uniquely determined up to $\nu$-equivalence.

The following are a few nontrivial examples.
\begin{example}[superposition]\label{ex:superposition}
Let $n$ be an integer greater than $1$, and take $\R^n$ as $X$.
Let $m$ be the Lebesgue measure on $\R^n$, and $(\cdot,\cdot)_{\R^n}$, the standard inner product of $\R^n$.
We denote the set of all $C^\infty$-functions on $\R^n$ with compact supports by $C_c^\infty(\R^n)$.
For $f,g\in C_c^\infty(\R^n)$, define
\[
\cE(f,g)=\frac12\int_{\R^n}(\nabla f(\bfx,y),\nabla g(\bfx,y))_{\R^n}\,d\bfx\,dy
+\frac12\int_{\R^{n-1}} \sum_{k=1}^{n-1}\frac{\partial f}{\partial x_k}(\bfx,0)\frac{\partial g}{\partial x_k}(\bfx,0)\,d\bfx,
\]
where $\bfx\in\R^{n-1}$, $y\in\R$, and $d\bfx$ and $dy$ represent the Lebesgue measures on $\R^{n-1}$ and $\R$, respectively.
Then, $(\cE,C_c^\infty(\R^n))$ is closable on $L^2(\R^n;m)$.
Its closure, denoted by $(\cE,\cF)$, is a strong local regular Dirichlet form on $L^2(\R^n;m)$.
The mutual energy measure $\mu_{\la f,g\ra}$ for $f,g\in\cF$ is described as
\[
  d\mu_{\la f,g\ra}=(\nabla f(\bfx,y),\nabla g(\bfx,y))_{\R^n}\,dm
  +\sum_{k=1}^{n-1}\frac{\partial \tilde f}{\partial x_k}(\bfx,0)\frac{\partial \tilde g}{\partial x_k}(\bfx,0)\,d\bfx\otimes\delta_0(dy),
\]
where $\dl_0$ is the Dirac measure at $0$, and $\tilde f$ and $\tilde g$ denote quasi-continuous modifications of $f$ and $g$, respectively.
Then, we can take $\nu=m+d\bfx\otimes\delta_0(dy)$ as an m.\,e.\,d.\,m.
We will show that the pointwise index $p(\bfx,y)$ $((\bfx,y)\in\R^n)$ and the index $p$ are given by
\begin{equation}\label{eq:pz}
p(\bfx,y)=\begin{cases}
n&\text{if $y\ne0$}\\
n-1&\text{if $y=0$}
\end{cases}
\quad\nu\text{-a.e.}
\end{equation}
and $p=n$.
For any finite number of functions $f_1,\dots,f_N$ in $\cF$, we have
\[
\frac{d\mu_{\la f_i,f_j\ra}}{d\nu}(\bfx,y)=
\begin{cases}
\displaystyle\left(\nabla f_i(\bfx,y),\nabla f_j(\bfx,y)\right)_{\R^n}&\text{if }y\ne0\\
\displaystyle\sum_{k=1}^{n-1}\frac{\partial \tilde f}{\partial x_k}(\bfx,0)\frac{\partial \tilde g}{\partial x_k}(\bfx,0)&\text{if }y=0
\end{cases}
\quad\nu\text{-a.e.}
\]
Accordingly,
\[
\left(\frac{d\mu_{\la f_i,f_j\ra}}{d\nu}(\bfx,y)\right)_{i,j=1}^N=
\begin{cases}
A(\bfx,y)\,\tr{\!A(\bfx,y)}&\text{if }y\ne0\\
B(\bfx)\,\tr{\!B(\bfx)}&\text{if }y=0,
\end{cases}
\]
where
$A(\bfx,y)$ is an $(N,n)$-matrix whose $(i,k)$-component is $(\partial{f_i}/\partial x_k)(\bfx,y)$, and $B(\bfx)$ is an $(N,n-1)$-matrix whose $(i,k)$-component is $(\partial{\tilde f_i}/\partial x_k)(\bfx,0)$.
Therefore, \Eq{pz} holds with ``$=$'' replaced by ``$\le$''. 
For $R>0$, take $f_1,\dots,f_n\in\cF$ such that $f_1(\bfx,y)=x_1,\dots, f_{n-1}(\bfx,y)=x_{n-1}, f_n(\bfx,y)=y$ on $\{|(\bfx,y)|_{\R^n}<R\}$, with $\bfx=(x_1,\dots,x_{n-1})\in\R^{n-1}$ and $y\in\R$.
Then, it is easy to see that
\[
\rank\left(\frac{d\mu_{\la f_i,f_j\ra}}{d\nu}(\bfx,y)\right)_{i,j=1}^n=
\begin{cases}
n&\text{if }y\ne0\\
n-1&\text{if }y=0
\end{cases}
\quad \text{$\nu$-a.e.\ on $\left\{|(\bfx,y)|_{\R^n}<R\right\}$.}
\]
Therefore, \Eq{pz} holds with ``$=$'' replaced by ``$\ge$''.
\end{example}
\begin{example}[fractals]
The construction of canonical Dirichlet forms on (self-similar) fractals has been studied extensively.
It is not easy to determine the exact value of the index; one of the reasons is that the energy measures do not have simple expressions.
According to \cite{Hi11p}, the index of a Dirichlet form associated with a regular harmonic structure on a post-critically finite, self-similar connected set is always $1$, and that of the Dirichlet form corresponding to Brownian motion on a class of generalized Sierpinski carpets is dominated by its spectral dimension, which does not exceed the Hausdorff dimension.
See \cite{Hi11p} and the references therein for further details.
\end{example}
\section{Measurable Riemannian structure}
We retain the general notations used in the previous section.
For $r=0,1,\dots,p$, set 
\[
X(r)=\{x\in X\mid p(x)=r\}.
\]
From \cite[Proposition~2.11]{Hi10}, $\nu(X(0))=0$.
In particular, $p=0$ if and only if $\cE\equiv0$.
Hereafter, we assume that the index $p$ of $(\cE,\cF)$ is finite and greater than $0$.
Denote the $p$~direct products of $\cF$ by $\cF^p$, and equip $\cF^p$ with the product topology.
We define subsets $\cG$ and $\hat \cG$ of $\cF^p$ by
\begin{align*}
\cG&=\left\{(g_1,\dots,g_p)\in\cF^p\;\vrule\;\;\parbox{0.53\textwidth}{For $\nu$-a.e.\,$x\in X$, the matrix
$\left(\dfrac{d\mu_{\la g_i,g_j\ra}}{d\nu}(x)\right)_{i,j=1}^{p(x)}$
of size $p(x)$ is invertible}\right\},\\
\hat\cG&=\left\{(g_1,\dots,g_p)\in\cG\;\vrule\;\;
\text{For every $i=1,\dots,p$, $\mu_{\la g_i\ra}$ is an m.\,e.\,d.\,m.}\right\}.
\end{align*}
The determination of these sets is independent of the choice of m.\,e.\,d.\,m.~$\nu$.
\begin{theorem}\label{th:main1}
Sets $\cG$ and $\hat\cG$ are dense in $\cF^p$.
\end{theorem}
\begin{proof}
The idea of the proof is based on that of \cite[Proposition~2.7]{Hi10}.
Take a c.o.n.s.\ $\{f_i\}_{i=1}^\infty$ of $\cF$.
Since the finite Borel measure $\sum_{i=1}^\infty 2^{-i}\mu_{\la f_i\ra}$ is an m.\,e.\,d.\,m.\ of $(\cE,\cF)$ from \cite[Lemma~2.3]{Hi10}, we may take $\sum_{i=1}^\infty 2^{-i}\mu_{\la f_i\ra}$ as $\nu$.
Let $\{\cB_n\}_{n=1}^\infty$ be a sequence of $\sg$-fields on $X$ such that $\cB_1\subset \cB_2\subset\cB_3\subset\cdots$, $\sg(\cB_n;n\in\N)$ is equal to the Borel $\sg$-field on $X$, and each $\cB_n$ is generated by a finite number of Borel subsets of $X$.
For each $n\in\N$, $\cB_n$ is determined by a partition of $X$ consisting of finitely many disjoint Borel sets $B_n^1,\dotsc,B_n^{M_n}$ for some $M_n\in \N$.
For each $i,j\in\N$, the Radon--Nikodym derivative $Z_n^{i,j}$ of $\mu_{\la f_i,f_j\ra}|_{\cB_n}$ with respect to $\nu|_{\cB_n}$ is defined as
\[
  Z_n^{i,j}(x)=\sum_{\a=1}^{M_n}\frac{\mu_{\la f_i,f_j\ra}(B_n^\a)}{\nu(B_n^\a)}\cdot 1_{B_n^\a}(x),
  \quad x\in X,
\]
where $0/0:=1$ by convention.
We define
\begin{equation*}
  X'=\{x\in X\mid \mbox{For every $i,j\in \N$, $Z_n^{i,j}(x)$ converges as $n\to\infty$}\}.
\end{equation*}
From the martingale convergence theorem, $\nu(X\setminus X')=0$.
For each $i,j\in \N$, we define
\begin{equation*}
  Z^{i,j}(x)=
  \begin{cases}
  \lim_{n\to\infty}Z_n^{i,j}(x) & \mbox{if }x\in X' \\
  0 & \mbox{if }x\in X\setminus X'.
  \end{cases}
\end{equation*}
Then, $Z^{i,j}$ is a Borel measurable representative of $d\mu_{\la f_i,f_j\ra}/d\nu$.
From \cite[Lemma~2.6]{Hi10}, $\bigl(Z^{i,j}(x)\bigr)_{i,j=1}^N$ is a nonnegative definite symmetric matrix for all $x\in X$ and $N\in\N$.
From \cite[Proposition~2.10]{Hi10}, for $\nu$-a.e.\,$x\in X$,
\begin{equation}\label{eq:px}
  p(x)=\sup_{N\in\N}\;\rank\,\bigl(Z^{i,j}(x)\bigr)_{i,j=1}^N.
\end{equation}
We may assume that the above identity holds for all $x\in X$ by redefining $p(x)$ as the right-hand side of \Eq{px}.

Let $\ell_2$ denote the usual $\ell_2$ space consisting of all real square-summable sequences.
The canonical inner product of $\ell_2$ is denoted by $(\cdot,\cdot)_{\ell_2}$.
Fix a finite Borel measure $\kp$ on $\ell_2$ such that the following properties hold:
\begin{equation}\label{eq:kappa}
\text{$\supp \kp=\ell_2$ and $\kp(L)=0$ for any proper closed subspace $L$ of $\ell_2$.}
\end{equation}
For example, it suffices to take a nondegenerate Gaussian measure on $\ell_2$ as $\kp$.

We define a map $\Psi\colon \ell_2\to\cF$ as
\[
  \Psi(\bfa)=\sum_{i=1}^\infty a_i 2^{-i/2}f_i
  \quad \text{for}\ \ \bfa=(a_i)_{i=1}^\infty\in\ell_2,
\]
where the limit on the right-hand side is taken in the topology of $\cF$.
It is easy to see that $\Psi$ is a contraction map and that $\Psi(\ell_2)$ is dense in $\cF$.

For $\bfa=(a_i)_{i=1}^\infty\in\ell_2$ and $N\in\N$, set $g_N(x):=\sum_{i=1}^N a_i 2^{-i/2}f_i(x)$ for $x\in X$ and $g:=\Psi(\bfa)$.
Since $\lim_{N\to\infty}g_N= g$ in $\cF$, from \cite[Lemma~2.5~(ii)]{Hi10}, we have that ${d\mu_{\la g_N\ra}}/{d\nu}$ converges to ${d\mu_{\la g\ra}}/{d\nu}$ as $N\to\infty$ in $L^1(X;\nu)$.
On the other hand, 
\begin{equation*}
\frac{d\mu_{\la g_N\ra}}{d\nu}(x)=\sum_{i,j=1}^N a_ia_j 2^{-(i+j)/2}Z^{i,j}(x) \text{ for $\nu$-a.e.\,} x\in X
\end{equation*}
and the right-hand side is absolutely convergent as $N\to\infty$ for each $x\in X$ from \cite[Eq.~(2.9)]{Hi10}.
Therefore, 
\begin{equation}
\frac{d\mu_{\la g\ra}}{d\nu}(x)=\lim_{N\to\infty}\frac{d\mu_{\la g_N\ra}}{d\nu}(x)
\quad \text{for $\nu$-a.e.\,$x\in X$}.
\end{equation}
For $x\in X$, let
\[
\Phi_x(\bfa,\bfb):=\sum_{i,j=1}^\infty a_i b_j 2^{-(i+j)/2}Z^{i,j}(x)
\quad \text{for $\bfa=(a_i)_{i=1}^\infty\in\ell_2$ and $\bfb=(b_i)_{i=1}^\infty\in\ell_2$}.
\]
As seen from \cite[p.~275]{Hi10}, $\Phi_x$ is a bounded symmetric bilinear form on $\ell_2$, which implies that there exists a bounded symmetric operator $A_x$ on $\ell_2$ such that $\Phi_x(\bfa,\bfb)=(\bfa,A_x\bfb)_{\ell_2}$ for every $\bfa,\bfb\in\ell_2$.
Moreover, 
\begin{equation}\label{eq:version}
\text{$\Phi_{\cdot}(\bfa,\bfb)$ is a $\nu$-version of $\frac{d\mu_{\la\Psi(\bfa),\Psi(\bfb)\ra}}{d\nu}(\cdot)$},
\end{equation}
and $\ker A_x=\{\bfa\in\ell_2\mid \Phi_x(\bfa,\bfa)=0\}$.
We denote $\ker A_x$ by $N_x$.
Then, we have the following identity.
\begin{lem}\label{lem:dim}
$\dim (\ell_2/N_x)=p(x)$.
\end{lem}
\begin{proof}
This lemma does not require the finiteness of $p$.
Let $\ell_0$ be a subspace of $\ell_2$ defined by
\[
\ell_0=\{\bfa=(a_i)_{i=1}^\infty\mid \mbox{$\bfa$ is a real sequence and $a_i=0$ except for finitely many $i$}\}.
\]
From \cite[Lemma~3.1]{Hi10}, $\dim(\ell_0/(N_x\cap \ell_0))=p(x)$.
Since $\ell_0$ is dense in $\ell_2$, $\ell_0/(N_x\cap \ell_0)$ is densely imbedded in $\ell_2/N_x$ with respect to the quotient topologies.
Therefore, $\dim(\ell_0/(N_x\cap \ell_0))=\dim(\ell_2/N_x)$.
\end{proof}
Let us return to the proof of \Thm{main1}.
We denote the $p$-direct product of $(\ell_2,\kp)$ by $((\ell_2)^p,\kp^{\otimes p})$ and define
\[
C=\left\{(x,\bfa^{(1)},\dots,\bfa^{(p)})\in X\times (\ell_2)^p\;\;\vrule\;\;\parbox{0.35\textwidth}{$p(x)\ge1$ and the matrix $\left(\Phi_x(\bfa^{(i)},\bfa^{(j)})\right)_{i,j=1}^{p(x)}$ is invertible}\right\}.
\]
Since $C\cap(X(r)\times(\ell_2)^p)$ is Borel measurable for each $r=1,\dots,p$, so is $C$.

Let $x\in X(1)$.
Then, $(x,\bfa^{(1)},\dots,\bfa^{(p)})\notin C$ if and only if $\bfa^{(1)}\in N_x$.
Since $\dim (\ell_2/N_x)=p(x)=1$ from \Lem{dim}, $N_x$ is a proper closed subspace of $\ell_2$. 
Therefore, $\kp(N_x)=0$ from \Eq{kappa} and
\begin{equation}\label{eq:null1}
(\nu\otimes\kp^{\otimes p})\left((X(1)\times(\ell_2)^p)\setminus C\right)=0.
\end{equation}
Let $2\le r\le p$ and $x\in X(r)$.
Then, $(x,\bfa^{(1)},\dots,\bfa^{(p)})\notin C$ if and only if,
\begin{align*}
&\bfa^{(1)}\in N_x\\
&\mbox{or }\bfa^{(2)}\in N_{x,\bfa^{(1)}}:=\mbox{the linear span of }\bigl(N_x\cup\{\bfa^{(1)}\}\bigr)\\
&\mbox{or }\bfa^{(3)}\in N_{x,\bfa^{(1)},\bfa^{(2)}}:=\mbox{the linear span of }\bigl(N_x\cup\{\bfa^{(1)},\bfa^{(2)}\}\bigr)\\
&\mbox{or }\cdots\\
&\mbox{or }\bfa^{(r)}\in N_{x,\bfa^{(1)},\dots,\bfa^{(r-1)}}:=\mbox{the linear span of }\bigl(N_x\cup\{\bfa^{(1)},\dots,\bfa^{(r-1)}\}\bigr).
\end{align*}
Since $\dim (\ell_2/N_x)=p(x)=r$, we have $\dim(\ell_2/N_{x,\bfa^{(1)},\dots,\bfa^{(s)}})\ge1$ for each $s=1,\dots,r-1$.
In particular, $N_{x,\bfa^{(1)},\dots,\bfa^{(s)}}$ is a proper closed subspace of $\ell_2$.
Therefore, $\kp(N_{x,\bfa^{(1)},\dots,\bfa^{(s)}})=0$, which implies that
\begin{equation}\label{eq:null2}
(\nu\otimes\kp^{\otimes p})((X(r)\times(\ell_2)^p)\setminus C)=0.
\end{equation}
From \Eq{null1}, \Eq{null2}, and the equality $\nu(X(0))=0$, we have
\begin{equation}\label{eq:nulla}
(\nu\otimes\kp^{\otimes p})((X\times(\ell_2)^p)\setminus C)=0.
\end{equation}
Next, we define
\[
\hat C=\left\{(x,\bfa^{(1)},\dots,\bfa^{(p)})\in X\times (\ell_2)^p\;\;\vrule\;\Phi_x(\bfa^{(i)},\bfa^{(i)})>0\text{ for every }i=1,\dots,p\right\}.
\]
Then, $(x,\bfa^{(1)},\dots,\bfa^{(p)})\notin \hat C$ if and only if $\bfa^{(i)}\in N_x$ for some $i=1,\dots,p$.
Since for every $x\in X$, $\kp^{\otimes p}(\{(\bfa^{(1)},\dots,\bfa^{(p)})\in (\ell_2)^p\mid (x,\bfa^{(1)},\dots,\bfa^{(p)})\notin\hat C\})=0$ from \Lem{dim} and \Eq{kappa}, we have 
\begin{equation}\label{eq:nullb}
  (\nu\otimes \kp^{\otimes p})\left((X\times(\ell_2)^p)\setminus \hat C\right)=0.
\end{equation}
From \Eq{version}, \Eq{nulla}, \Eq{nullb}, and Fubini's theorem, for $\kp^{\otimes p}$-a.e.\,$(\bfa^{(1)},\dots,\bfa^{(p)})\in (\ell_2)^p$, we have that
\[
\left(\frac{d\mu_{\la \Psi(\bfa^{(i)}),\Psi(\bfa^{(j)})\ra}}{d\nu}(x)\right)_{i,j=1}^{p(x)} \text{ is invertible}
\] 
and 
\[
\text{$\mu_{\la \Psi(\bfa^{(i)})\ra}$ is an m.\,e.\,d.\,m.\ for all $i=1,\dots,p$}
\]
for $\nu$-a.e.\,$x\in X$.
Since $(\Psi(\ell_2))^p$ is dense in $\cF^p$, we obtain that $\hat \cG$, along with $\cG$, is dense in $\cF^p$.
\end{proof}
We fix $\bfg=(g_1,\dots,g_p)\in\cG$ and write $Z_\bfg^{i,j}$ for ${d\mu_{\la g_i,g_j\ra}}/{d\nu}$ for $i,j=1,\dots,p$.
Let $r\in\{1,\dots,p\}$.
We denote the matrix
$\left(Z_\bfg^{i,j}(x)\right)_{\!i,j=1}^r$
by $Z_{\bfg,r}(x)$ for $x\in X$.
For $\nu$-a.e.\,$x\in X(r)$, $Z_{\bfg,r}(x)$ is invertible by definition of $\cG$.

Let $f\in \cF$.
We define a $\nu$-measurable $\R^r$-valued function $\bfu_r$ on $X$ as 
\begin{equation}\label{eq:ur}
\bfu_r=\left[\begin{array}{c}
d\mu_{\la f,g_1\ra}/d\nu\\\vdots\\d\mu_{\la f,g_r\ra}/d\nu\end{array}\right].
\end{equation}
\begin{lem}\label{lem:Kr}
 For each $r=1,\dots,p$,
 \begin{equation}\label{eq:Kr}
\frac{d\mu_{\la f\ra}}{d\nu}=\tr\bfu_r Z_{\bfg,r}^{-1}\bfu_r\quad \nu\text{-a.e. on }X(r).
\end{equation}
\end{lem}
\begin{proof}
From the definition of $X(r)$, the $\R^{(r+1)\times(r+1)}$-valued function
$\renewcommand{\arraystretch}{1.2}\left(\begin{array}{c|c}Z_{\bfg,r}&\bfu_r\\\hline\tr\bfu_r&d\mu_{\la f\ra}/d\nu\end{array}\right)$ is not invertible $\nu$-a.e.\ on $X(r)$.
Then, on $X(r)$,
\begin{align*}
0&=\renewcommand{\arraystretch}{1.2}\det\left(\left(\begin{array}{c|c}\strut Z_{\bfg,r}&\bfu_r\\\hline\strut\tr\bfu_r&d\mu_{\la f\ra}/d\nu\end{array}\right)
\left(\begin{array}{c|c}Z_{\bfg,r}^{-1}&-Z_{\bfg,r}^{-1}\bfu_r\\\hline\bfz&1\end{array}\right)\right)\\
&=\renewcommand{\arraystretch}{1.2}\det\left(\begin{array}{c|c}I&\bfz\\\hline\tr\bfu_r Z_{\bfg,r}^{-1}&-\tr\bfu_r Z_{\bfg,r}^{-1}\bfu_r+d\mu_{\la f\ra}/d\nu\end{array}\right)\\
&=-\tr\bfu_r Z_{\bfg,r}^{-1}\bfu_r+\frac{d\mu_{\la f\ra}}{d\nu}\quad \nu\text{-a.e.},
\end{align*}
that is, \Eq{Kr} holds.
\end{proof}
We denote a quasi-continuous modification of $f$ by $\tilde f$.
The following theorem is a generalization of \cite[Theorem~5.4]{Hi10}.
\begin{theorem}[analogue of differentials of functions in $\cF$]\label{th:main2}
There exists a $\nu$-measurable $\R^p$-valued function $\nabla\!_\bfg f=\tr(\partial^{(1)}f,\dots,\partial^{(p)}f)$ on $X$ such that the following hold: 
\begin{equation}\label{eq:zero}
\text{For $\nu$-a.e.\,$x$, $\partial^{(i)}f(x)=0$ for all $i>p(x)$},
\end{equation}
and
\begin{equation}\label{eq:taylor}
\tilde f(y)-\tilde f(x)=\sum_{i=1}^{p(x)} \partial^{(i)}f(x)\bigl(\tilde g_i(y)-\tilde g_i(x)\bigr)+R_x(y),\quad y\in X,
\end{equation}
where $R_x(\cdot)\in\cFloc$ is negligible at $x$ in the sense that 
\begin{equation}\label{eq:neg}
\frac{d\mu_{\la R_x\ra}}{d\nu}(x)=0
\quad\text{for $\nu$-a.e.\,$x$}.
\end{equation}
The function $\nabla\!_\bfg f$ is uniquely determined up to $\nu$-equivalence.
Moreover, the following identity holds\rom{:}
\begin{equation}\label{eq:rep}
\cE(f,h)=\frac12\int_X\bigl( Z_\bfg\,\nabla\!_\bfg f,\nabla\!_\bfg h\bigr)_{\R^p}\,d\nu,\quad f,h\in\cF.
\end{equation}
\end{theorem}
In the statement above, the precise meaning of $({d\mu_{\la R_x\ra}}/{d\nu})(x)$ in \Eq{neg} is as follows.
First, we fix $\nu$-versions of $d\mu_{\la f\ra}/d\nu$, $d\mu_{\la f,g_i\ra}/d\nu$, and $d\mu_{\la g_i,g_j\ra}/d\nu$ for $i,j=1,\dots,p$.
When $R_x(y)$ is determined by \Eq{taylor} as a function of $y$ for fixed $x\in X$, $({d\mu_{\la R_x\ra}}/{d\nu})(x)$ is defined by letting $y=x$ in the natural $\nu$-version of $({d\mu_{\la R_x\ra}}/{d\nu})(y)$, that is,
\begin{equation}\label{eq:Rx}
\frac{d\mu_{\la R_x\ra}}{d\nu}(x)=
 \frac{d\mu_{\la f\ra}}{d\nu}(x)-2\sum_{i=1}^{p(x)} \partial^{(i)}f(x)\,\frac{d\mu_{\la f,g_i\ra}}{d\nu}(x)+\sum_{i,j=1}^{p(x)} \partial^{(i)}f(x)\,\partial^{(j)}f(x)\,\frac{d\mu_{\la g_i,g_j\ra}}{d\nu}(x).
\end{equation}
Here, we used \Eq{constant}.

From this theorem, we can regard the map $\bfg\colon X\to\R^p$ as a type of local coordinate system (although it is not necessarily injective), and $Z_\bfg(x)$ and $p(x)$ as the Riemannian metric and the dimension of the virtual tangent space of $x\in X$, respectively. 
Note that $Z_\bfg(\cdot)$ and $p(\cdot)$ make sense only $\nu$-almost everywhere, not everywhere.

In typical examples of Dirichlet forms $(\cE,\cF)$ on self-similar fractals~$X$, $\cF$ is characterized by a Besov space (see, e.g., \cite{Jo96,Kum00,Gr03}).
Even if $X$ is imbedded in the Euclidean space, functions in $\cF$ are generally far from smooth in the usual sense; 
nevertheless, \Thm{main2} implies that the infinitesimal behaviors of functions in $\cF$ can be described by those of representatives of $\cF$.
\begin{proof}[Proof of \Thm{main2}]
First, we prove \Eq{neg} and \Eq{rep} for a suitable function $\nabla\!_\bfg f$.
We define $\nabla\!_\bfg f=\tr(\partial^{(1)}f,\dots,\partial^{(p)}f)$ by
\begin{equation}\label{eq:gradient}
\left[\begin{array}{c}\partial^{(1)}f\\\vdots\\\partial^{(r)}f\end{array}\right]
=Z_{\bfg,r}^{-1}\bfu_r\text{ and }
\left[\begin{array}{c}\partial^{(r+1)}f\\\vdots\\\partial^{(p)}f\end{array}\right]
=\left[\begin{array}{c}0\\\vdots\\0\end{array}\right]\quad
\text{on $X(r)$, for $r=0,1,\dots,p$},
\end{equation}
where $\bfu_r$ is provided in \Eq{ur}.
Fix $r\in\{1,\dots,p\}$.
From \Eq{Rx}, for $\nu$-a.e.\,$x\in X(r)$,
\begin{align*}
\frac{d\mu_{\la R_x\ra}}{d\nu}(x)
&=\frac{d\mu_{\la f\ra}}{d\nu}(x)-2\sum_{i=1}^{r} \partial^{(i)}f(x)\frac{d\mu_{\la f,g_i\ra}}{d\nu}(x)+\sum_{i,j=1}^{r} \partial^{(i)}f(x)\,\partial^{(j)}f(x)\frac{d\mu_{\la g_i,g_j\ra}}{d\nu}(x)\\
&=\frac{d\mu_{\la f\ra}}{d\nu}(x)-2(Z_{\bfg,r}^{-1}\bfu_r,\bfu_r)_{\R^r}(x)+\left(\tr(Z_{\bfg,r}^{-1}\bfu_r)Z_{\bfg,r}(Z_{\bfg,r}^{-1}\bfu_r)\right)(x)\\
&=\frac{d\mu_{\la f\ra}}{d\nu}(x)-\left(\tr{\bfu_r}Z_{\bfg,r}^{-1}\bfu_r\right)(x).
\end{align*}
From \Lem{Kr}, the last term vanishes $\nu$-a.e.\ on $X(r)$.
Therefore, \Eq{neg} holds.
To prove \Eq{rep}, we may assume that $f=h$ since both sides of \Eq{rep} are bilinear in $f$ and $h$.
For $r\in\{1,\dots,p\}$, 
\begin{equation}
\bigl( Z_\bfg\,\nabla\!_\bfg f,\nabla\!_\bfg f\bigr)_{\R^p}
=\bigl(\bfu_r,Z_{\bfg,r}^{-1}\bfu_r\bigr)_{\R^r}
=\frac{d\mu_{\la f\ra}}{d\nu} \quad\text{$\nu$-a.e.\ on $X(r)$}
\end{equation}
from \Lem{Kr}, which implies that
\begin{equation}
\frac12\int_X\bigl( Z_\bfg\,\nabla\!_\bfg f,\nabla\!_\bfg f\bigr)_{\R^p}d\nu
=\frac12\int_X\frac{d\mu_{\la f\ra}}{d\nu}\,d\nu
=\frac12\mu_{\la f\ra}(X).
\end{equation}
From \Eq{total}, \Eq{rep} holds.

In order to prove the uniqueness of $\nabla\!_\bfg f$, suppose that $\hat\nabla\!_\bfg f=\tr(\hat\partial^{(1)}f,\dots,\hat\partial^{(p)}f)$ and $\hat R_x\in\cFloc$ ($x\in X)$ satisfy \Eq{zero}, \Eq{taylor}, and \Eq{neg}, with $\partial^{(i)}f$ and $R_x$ replaced by $\hat\partial^{(i)}f$ and $\hat R_x$, respectively.
Then,
\[
0=\sum_{i=1}^{p(x)}\left(\partial^{(i)}f(x)-\hat\partial^{(i)}f(x)\right)\left(\tilde g_i(y)-\tilde g_i(x)\right)+\left(R_x(y)-\hat R_x(y)\right),\quad y\in X.
\]
We denote $\partial^{(i)}f(x)-\hat\partial^{(i)}f(x)$ by $h_i(x)$ for $i=1,\dots,p$, and let $R'_x=R_x-\hat R_x$.
For $r\in\{1,\dots,p\}$ and $\nu$-a.e.\,$x\in X(r)$,
\begin{align*}
0&=\mu_{\left\la \sum_{i=1}^r h_i(x)g_i(\cdot)+R'_x(\cdot) \right\ra}\\
&=\sum_{i,j=1}^r h_i(x)h_j(x)\mu_{\la g_i,g_j\ra}+2\sum_{i=1}^r h_i(x)\mu_{\la g_i,R'_x\ra}+\mu_{\la R'_x\ra}.
\end{align*}
Then, since $({d\mu_{\la R'_x\ra}}/{d\nu})(x)=0$ and $({d\mu_{\la g_i,R'_x\ra}}/{d\nu})(x)=0$ for $\nu$-a.e.\,$x$,
\begin{align*}
0=\sum_{i,j=1}^r h_i(x)h_j(x)\frac{d\mu_{\la g_i,g_j\ra}}{d\nu}(x)
=\sum_{i,j=1}^r h_i(x)h_j(x)Z_\bfg^{i,j}(x)
\end{align*}
for $\nu$-a.e.\ $x\in X(r)$, which implies that $h_i(x)=0$ for all $i=1,\dots,r$.
Since $\partial^{(i)}f=\hat\partial^{(i)}f=0$ on $X(r)$ for $i=r+1,\dots,p$ from \Eq{zero}, we obtain the uniqueness of $\nabla\!_\bfg f$.
\end{proof}
\begin{remark}
From the proof of \Thm{main2}, we can define $\nabla\!_\bfg f$ for $f\in \cFloc$ in the natural way.
Also, $\nabla\!_\bfg$ satisfies the derivation property: for $f_1,\dots,f_k\in\cFloc$ and $\Psi\in C^1(\R^k)$, 
\begin{equation}
\nabla\!_\bfg \bigl(\Psi(f_1,\dots,f_k)\bigr)
=\sum_{i=1}^k \frac{\partial\Psi}{\partial x_i}(f_1,\dots,f_k)\nabla\!_\bfg f_i.
\end{equation}

\end{remark}
\section{Application to stochastic analysis}
In this section, we discuss an application to stochastic analysis.
We introduce some necessary notations, following Chapter~5 of \cite{FOT}.
Let us recall that the diffusion process $\{X_t\}$ associated with $(\cE,\cF)$ is defined on a filtered probability space $(\Omega,\cF_\infty,P,\{P_x\}_{x\in X_\Delta},\{\cF_t\}_{t\in[0,\infty)})$.
We denote the expectation with respect to $P_x$ by $E_x$.
Let $\cM$ be the set of all finite c\`adl\`ag additive functionals $M$ such that for each $t>0$, $E_x[M_t^2]<\infty$ and $E_x[M_t]=0$ for q.e.\,$x\in X$. 
By the strong locality of $(\cE,\cF)$, every $M\in\cM$ is, in fact, a continuous additive functional.
For $M\in\cM$, we denote its quadratic variation by $\la M\ra$, which is a positive continuous additive functional, and the Revuz measure of $\la M\ra$ by $\mu_{\la M\ra}$.
The measure $\mu_{\la M\ra}$ is also called the energy measure of $M$.
A signed measure $\mu_{\la M,L\ra}$ on $X$ for $M,L\in \cM$ is defined as
\[
 \mu_{\la M,L\ra}:=\frac12(\mu_{\la M+L\ra}-\mu_{\la M\ra}-\mu_{\la L\ra}).
\]
For $M\in \cM$, its energy $e(M)$ is defined as
\[
  e(M)=\sup_{t>0}\frac1{2t}\int_X E_x[M_t^2]\,m(dx)\,(\le+\infty).
\]
We set $\maruM=\{M\in\cM\mid e(M)<\infty\}$.
By setting
\[
e(M,L):=\frac12(e(M+L)-e(M)-e(L))\quad
\text{for }M,L\in\maruM,
\]
$\maruM$ becomes a Hilbert space with inner product $e(\cdot,\cdot)$.
For $M\in\maruM$ and $h\in L^2(X;\mu_{\la M\ra})$, the stochastic integral $h\bullet M\in\maruM$ is defined by the following characterization:
\[
  e(h\bullet M,L)=\frac12\int_X h\,d\mu_{\la M,L\ra}
  \quad\text{for every $L\in \maruM$}.
\]
We recall the following stochastic interpretation of the index.
\begin{theorem}[{\cite[Theorem~3.4]{Hi10}}]\label{th:martingaledim}
The index of $(\cE,\cF)$ coincides with the AF-martingale dimension of $\{X_t\}$ that is defined as the smallest number $p$ in $\Z_+\cup\{+\infty\}$ satisfying the following: There exists a sequence $\{M^{(k)}\}_{k=1}^p$ in $\maruM$ such that every $M\in\maruM$ has a stochastic integral representation
\begin{equation}\label{eq:si}
  M_t=\sum_{k=1}^p(h_k\bullet M^{(k)})_t,\quad t>0,\ P_x\mbox{-a.e.\ for q.e.\,}x\in X
\end{equation}
with some $h_k\in L^2(X;\mu_{\la M^{(k)}\ra})$ for $k=1,\dotsc,p$.
\end{theorem}
The theorem stated above is valid even if the index is $+\infty$.
Henceforth, we assume that the index $p$ is finite and greater than $0$, as in the previous section.
In order to state the main theorem in this section,
we make a slight generalization of the stochastic integrals to vector-valued functions/additive functionals.
As in the previous case, we fix an m.\,e.\,d.\,m.\ $\nu$ of $(\cE,\cF)$.
Let $n\in\N$ and $M_1,\dots,M_n\in\maruM$. We write $\bfM$ for $\tr(M_1,\dots,M_n)$ and $Z_\bfM$ for $\left({d\mu_{\la M_i,M_j\ra}}/{d\nu}\right)_{i,j=1}^n$. 
We remark that  $\mu_{\la M_i,M_j\ra}\ll \nu$ for any $i$ and $j$ from \cite[Lemma~3.2]{Hi10}.
Define
\begin{align*}
  L^2(X\to\R^n;\bfM)&=\left\{\bfh=\tr(h_1,\dots,h_n)\;\vrule\;\;\parbox{0.4\textwidth}{$h_i$ is a $\nu$-measurable real-valued funtion on $X$ for $i=1,\dots,n$, and $\displaystyle\int_X (\bfh,Z_\bfM\bfh)_{\R^n}\,d\nu<+\infty$}\,\right\},\\
  \hat L^2(X\to\R^n;\bfM)&=\left\{\bfh=\tr(h_1,\dots,h_n)\in L^2(X\to\R^n;\bfM)\;\vrule\;\;\parbox{0.25\textwidth}{$h_i\in L^2(X;\mu_{\la M_i\ra})$ for every $i=1,\dots,n$}\,\right\}.
\end{align*}
The definition of these spaces is independent of the choice of $\nu$.
For $\bfh,\bfh'\in L^2(X\to\R^n;\bfM)$, we set
\[
(\bfh,\bfh')_{\bfM}=\int_X (\bfh,Z_\bfM\bfh')_{\R^n}\,d\nu.
\]
 Any $\bfh=\tr(h_1,\dots,h_n)\in L^2(X\to\R^n;\bfM)$ can be approximated by elements $\{\bfh^{(k)}\}_{k=1}^\infty$ in $\hat L^2(X\to\R^n;\bfM)$ in the sense that 
$(\bfh-\bfh^{(k)},\bfh-\bfh^{(k)})_{\bfM}\to0$ as $k\to\infty$.
Indeed, it suffices to take 
\begin{equation}\label{eq:approx}
\bfh^{(k)}=\tr(h_1\cdot1_{A_k},\dots,h_n\cdot1_{A_k}),
\end{equation}
where $A_k=\{x\in X\mid (\bfh(x),Z_\bfM(x)\bfh(x))_{\R^n}\le k\}$.

For $\bfh=\tr(h_1,\dots,h_n)\in \hat L^2(X\to\R^n;\bfM)$, we define $\bfh\bullet\bfM\in\maruM$ as $\bfh\bullet\bfM=\sum_{i=1}^n h_i\bullet M_i$.
Since 
\begin{align*}
e(\bfh\bullet\bfM)
&=\sum_{i,j=1}^n e(h_i\bullet M_i,h_j\bullet M_j)\\
&=\frac12\sum_{i,j=1}^n\int_{X}h_i h_j \,d\mu_{\la M_i,M_j\ra}\\
&=\frac12(\bfh,\bfh)_{\bfM},
\end{align*}
we can define $\bfh\bullet\bfM\in\maruM$ for any $\bfh\in L^2(X\to\R^n;\bfM)$ by approximating $\bfh$ by elements in $\hat L^2(X\to\R^n;\bfM)$.
Then, the identity
$e(\bfh\bullet\bfM)=(\bfh,\bfh)_{\bfM}/2$
holds for all $\bfh\in L^2(X\to\R^n;\bfM)$.

For $f\in\cF$, let $M^{[f]}$ denote the element of $\maruM$ that appears in the Fukushima decomposition of the additive functional $\tilde f(X_t)-\tilde f(X_0)$ (cf.~\cite[Theorem~5.2.2]{FOT}). 
Its energy measure $\mu_{\la M^{[f]}\ra}$ coincides with $\mu_{\la f\ra}$ from \cite[Theorem~5.2.3]{FOT}.
In particular,
\begin{equation}\label{eq:fuku}
\mu_{\la M^{[f]},M^{[g]}\ra}=\mu_{\la f,g\ra}\quad
\text{for }f,g\in\cF.
\end{equation}

Fix $\bfg=(g_1,\dots,g_p)\in\cG$ and denote $\tr{\left(M^{[g_1]},\dots,M^{[g_p]}\right)}$ by $\bfM^{[\bfg]}$.
As in the previous section, we set $Z_\bfg^{i,j}={d\mu_{\la g_i,g_j\ra}}/{d\nu}$ for $i,j\in\{1,\dots,p\}$, and $Z_{\bfg,r}=\left(Z_\bfg^{i,j}\right)_{\!i,j=1}^r$ for $r=1,\dots,p$.
We write $Z_\bfg$ for $Z_{\bfg,p}$, which is identical to $Z_{\bfM^{[\bfg]}}$.
The following theorem is an improvement on a part of \Thm{martingaledim}.
\begin{theorem}\label{th:main3}
 Let $M\in\maruM$.
 Then, there exists $\bfh=\tr{(h_1,\dots,h_p)}\in L^2(X\to\R^p;\bfM^{[\bfg]})$ such that
 \begin{equation}\label{eq:cond1}
 \text{for $\nu$-a.e.\,$x$, $h_i(x)=0$ for all $i>p(x)$},  \end{equation}
and 
\begin{equation}\label{eq:cond2}
M=\bfh\bullet\bfM^{[\bfg]}.
\end{equation}
The function $\bfh$ is uniquely determined up to $\nu$-equivalence.
Moreover, if $M=M^{[f]}$ for some $f\in\cF$, then $\bfh$ is provided by $\nabla\!_\bfg f$.
\end{theorem}
In the proof of \Thm{martingaledim}, $M^{(i)}$'s in \Eq{si} are taken from $\maruM$.
\Thm{main3} shows that we can take $M^{(i)}$'s from a smaller set $\{M^{[f]}\mid f\in\cF\}$ at the expense of regarding the stochastic integral as that of vector-valued functions/additive functionals; we can no longer decompose it into the sum of scalar-valued stochastic integrals in general, at least in $\maruM$.
\begin{proof}[Proof of \Thm{main3}]
We write $v_i$ for $d\mu_{\la M,M^{[g_i]}\ra}/d\nu$ for $i=1,\dots,p$.
For $r\in\{1,\dots,p\}$, we set a $\nu$-measurable $\R^r$-valued function $\bfv_r$ on $X$ by 
$\bfv_r=\tr(v_1,\dots,v_r)$.
We define an $\R^p$-valued function $\bfh=\tr(h_1,\dots,h_p)$ on $X$ satisfying \Eq{cond1} by 
\begin{equation}\label{eq:h}
\left[\begin{array}{c}h_1\\\vdots\\h_r\end{array}\right]
=Z_{\bfg,r}^{-1}\bfv_r\text{ and }
\left[\begin{array}{c}h_{r+1}\\\vdots\\h_p\end{array}\right]
=\left[\begin{array}{c}0\\\vdots\\0\end{array}\right]\quad
\text{on $X(r)$, for $r=0,1,\dots,p$.}
\end{equation}
We note that $\bfh=\nabla\!_\bfg f$ if $M=M^{[f]}$ for some $f\in\cF$, in view of \Eq{ur}, \Eq{gradient}, and \Eq{fuku}.

In order to prove that $\bfh\in L^2(X\to\R^p;\bfM^{[\bfg]})$ and $M=\bfh\bullet\bfM^{[\bfg]}$, let us first consider the case that $M$ is described as $M=\ph\bullet M^{[f]}$ for some $f\in\cF$ and $\ph\in L^2(X;\mu_{\la f\ra})$.
Denote the set of all such $M$'s by $\maruM_{\#}$.
In this case, we have
\begin{align}\label{eq:eM}
\frac12\int_{X}\tr \bfh Z_{\bfg}\bfh\,d\nu
&=\frac12\sum_{r=1}^p\int_{X(r)}\tr\bfv_r Z_{\bfg,r}^{-1} \bfv_r\,d\nu\notag\\
&=\frac12\sum_{r=1}^p\int_{X(r)}\ph^2 \,\tr\bfu_r Z_{\bfg,r}^{-1} \bfu_r\,d\nu
\quad\text{(cf.\ $\bfu_r$ is defined in \Eq{ur})}\notag\\
&=\frac12\sum_{r=1}^p\int_{X(r)}\ph^2\frac{d\mu_{\la f\ra}}{d\nu}\,d\nu
\quad\text{(from \Lem{Kr})}\notag\\
&=\frac12\int_X \ph^2\,d\mu_{\la M^{[f]}\ra}\notag\\
&=e(M)<+\infty;
\end{align}
thus, $\bfh\in L^2(X\to\R^p;\bfM^{[\bfg]})$.
By approximating $\bfh$ by elements $\bfh^{(k)}=\tr(h^{(k)}_1,\dots,h^{(k)}_p)$ in $\hat L^2(X\to\R^p;\bfM^{[\bfg]})$ as in \Eq{approx}, we have
\begin{align}\label{eq:en}
&e(M-\bfh\bullet \bfM^{[\bfg]})\notag\\
&=\lim_{k\to\infty}e(M-\bfh^{(k)}\bullet \bfM^{[\bfg]})\notag\\
&=e(M)-\lim_{k\to\infty}2e(M,\bfh^{(k)}\bullet \bfM^{[\bfg]})+\lim_{k\to\infty}e(\bfh^{(k)}\bullet \bfM^{[\bfg]})\notag\\
&=e(M)-\lim_{k\to\infty}\sum_{i=1}^p \int_X h_i^{(k)}\,d\mu_{\langle M, M^{[g_i]}\rangle}
+\lim_{k\to\infty}\frac12\sum_{i,j=1}^p \int_X h_i^{(k)}h_j^{(k)}\,d\mu_{\langle M^{[g_i]}, M^{[g_j]}\rangle}\notag\\
&=e(M)-\lim_{k\to\infty} \int_X \sum_{i=1}^p 1_{A_k}h_i v_i\,d\nu+\lim_{k\to\infty}\frac12\int_X \sum_{i,j=1}^p 1_{A_k}h_i h_j Z_{\bfg}^{i,j}\,d\nu\notag\\
&=e(M)-\lim_{k\to\infty}\frac12\int_{A_k}\tr \bfh Z_{\bfg}\bfh\,d\nu\notag\\
&=e(M)-\frac12\int_{X}\tr \bfh Z_{\bfg}\bfh\,d\nu,
\end{align}
which vanishes from \Eq{eM}.
Therefore, \Eq{cond2} holds.

Let $\maruM_{\#\#}$ be the set of all $M\in\maruM$ described as a finite sum of additive functionals in $\maruM_{\#}$.
By linearity, \Eq{cond2} is true for $M\in\maruM_{\#\#}$.
From the observation that \Eq{en} is valid as long as $M\in \maruM$ and $\bfh\in L^2(X\to\R^p;\bfM^{[\bfg]})$ satisfy \Eq{h}, the identity
\begin{equation}\label{eq:isom}
e(M)=\frac12\int_{X}\tr \bfh Z_{\bfg}\bfh\,d\nu
=\frac12(\bfh,\bfh)_{\bfM^{[\bfg]}}
\end{equation}
holds for $M\in\maruM_{\#\#}$.
From the denseness of $\maruM_{\#\#}$ in $\maruM$ (cf.~\cite[Lemma~5.6.3]{FOT}) and the isometry~\Eq{isom}, identity~\Eq{isom} extends to all $M\in\maruM$.
By combining \Eq{en} with \Eq{isom}, Eq.~\Eq{cond2} holds for all $M\in\maruM$.

If another $\bfh'\in L^2(X\to\R^p;\bfM^{[\bfg]})$ satisfies \Eq{cond1} and \Eq{cond2} with $\bfh$ replaced by $\bfh'$, then $(\bfh-\bfh')\bullet\bfM^{[\bfg]}$ vanishes. 
From \Eq{isom}, $(\bfh-\bfh',\bfh-\bfh')_{\bfM^{[\bfg]}}=0$.
Since $Z_{\bfg,r}$ is invertible on $X(r)$ for each $r=1,\dots,p$, we obtain that $\bfh=\bfh'$ $\nu$-a.e.\ by taking \Eq{cond1} into consideration.
\end{proof}
\begin{acknowledgement}
The author thanks Professor Nobuyuki Ikeda for insightful discussions.
\end{acknowledgement}

\vspace{\baselineskip}
%

\end{document}